\documentclass[reqno]{amsproc}

\usepackage[margin=1in,marginparwidth=2cm, marginparsep=0.2cm]{geometry}
\usepackage{setspace, fullpage}
\usepackage{amsmath,amsthm,amscd,amssymb,bbm,mathrsfs,latexsym}
\usepackage[colorlinks,citecolor=red,pagebackref,hypertexnames=false]{hyperref}

\DeclareMathOperator{\Adj}{Adj}

\usepackage{mathtools}
\newcommand{\defeq}{\vcentcolon=}

\usepackage{tikz}
\usetikzlibrary{arrows,shapes}
\usetikzlibrary{decorations.markings}

\tikzstyle{white}=[circle,draw=black!100,fill=white!100,thick,inner sep=0pt,minimum size =2mm]
\tikzstyle{black}=[circle,draw=black!100,fill=black!100,thick,inner sep=0pt,minimum size =2mm]
\tikzstyle{graywhite}=[circle,draw=gray!100,fill=white!100,thick,inner sep=0pt,minimum size =2mm]

\tikzset{->-/.style={decoration={
  markings,
  mark=at position .5 with {\arrow[scale=1.3]{>}}},
  postaction={decorate}}
  }

\numberwithin{equation}{section}
\theoremstyle{plain}
\newtheorem{theorem}{{\bf Theorem}}[section]
\newtheorem{lemma}[theorem]{{\bf Lemma}}
\newtheorem{corollary}[theorem]{Corollary}

\theoremstyle{definition}
\newtheorem{definition}[theorem]{{\bf Definition}}

\theoremstyle{remark}
\newtheorem{remark}[theorem]{Remark}
\numberwithin{equation}{section}

\begin{document}

\title{The asymptotic enhanced negative type of finite ultrametric spaces}

\author{Ian Doust}
\address{School of Mathematics and Statistics, University of New South Wales, Sydney, New South Wales 2052, Australia}
\email{i.doust@unsw.edu.au}
\author{Stephen S\'{a}nchez}
\address{School of Mathematics and Statistics, University of New South Wales, Sydney, New South Wales 2052, Australia}
\email{stephen.sanchez@unsw.edu.au}
\author{Anthony Weston}
\address{Department of Mathematics and Statistics, Canisius College, Buffalo, NY 14208, USA}
\email{westona@canisius.edu}
\address{Department of Decision Sciences, University of South Africa, PO Box 392, UNISA 0003, South Africa}
\email{westoar@unisa.ac.za}



\begin{abstract}
Negative type inequalities arise in the study of embedding properties of metric spaces, 
but they often reduce to intractable combinatorial problems. In this paper we study more quantitative versions 
of these inequalities involving the so-called $p$-negative type gap. In particular, we focus our attention
on the class of finite ultrametric spaces which are important in areas such as phylogenetics and data mining.

Let $(X,d)$ be a given finite ultrametric space with minimum non-zero distance $\alpha$. Then the
$p$-negative type gap $\Gamma_{X}(p)$ of $(X,d)$ is positive for all $p \geq 0$.
In this paper we compute the value of the limit
\begin{eqnarray*}
\Gamma_{X}(\infty) & \defeq & \lim\limits_{p \rightarrow \infty} \frac{\Gamma_{X}(p)}{\alpha^{p}}.
\end{eqnarray*}
It turns out that this value is positive and it may be given explicitly by an elegant combinatorial formula.
On the basis of our calculations we are then able to characterize when $\Gamma_{X}(p)/ \alpha^{p}$ is constant on $[0, \infty)$.

The determination of $\Gamma_{X}(\infty)$ also leads to new, asymptotically sharp, families of enhanced $p$-negative
type inequalities for $(X,d)$. Indeed, suppose that $G \in (0, \Gamma_{X}(\infty))$. Then, for all sufficiently large $p$,
we have
\begin{eqnarray*}
\frac{G \cdot \alpha^{p}}{2} \left( \sum\limits_{k=1}^{n} |\zeta_{k}| \right)^{2}
+ \sum\limits_{j,i =1}^{n} d(z_{j},z_{i})^{p} \zeta_{j} \zeta_{i} & \leq & 0
\end{eqnarray*}
for each finite subset $\{ z_{1}, \ldots, z_{n} \} \subseteq X$
and each choice of real numbers $\zeta_{1}, \ldots, \zeta_{n}$ with $\zeta_{1} + \cdots + \zeta_{n} = 0$.
We note that these results do not extend to general finite metric spaces.
\end{abstract}

\maketitle

\section{Introduction and statement of the main results}\label{sec:1}

Since the early 1990s there has been renewed interest in embedding properties of negative type metrics.
One compelling reason for this is a fundamental link to the design of algorithms for cut problems \cite{Ar1, Cha, Ar2}.
The subject of this paper is the important and closely related class of \textit{strict} negative type metrics.
The first systematic treatment of strict negative type metrics appears in the elegant papers of Hjorth
\textit{et al}.\ \cite{Hj1, Hj2}. Informally, a metric space is of strict negative type when all of the non-trivial
negative type inequalities for the space are strict. More precisely, we have the following definition.

\begin{definition}\label{neg:type} Let $p \geq 0$ and let $(X,d)$ be a metric space.
\begin{enumerate}
\item $(X,d)$ has $p$-{\textit{negative type}} if and only if for all integers $n > 1$,
all finite subsets $\{z_{1}, \ldots , z_{n} \} \subseteq X$, and all scalar $n$-tuples
$\boldsymbol{\zeta} = (\zeta_{1}, \ldots, \zeta_{n}) \in \mathbb{R}^{n}$ that satisfy
$\zeta_{1} + \cdots + \zeta_{n} = 0$, we have:
\begin{eqnarray}\label{p:neg}
\sum\limits_{j,i =1}^{n} d(z_{j},z_{i})^{p} \zeta_{j} \zeta_{i} & \leq & 0.
\end{eqnarray}

\item $(X,d)$ has \textit{strict} $p$-{\textit{negative type}} if and only if it has $p$-negative type
and the inequalities (\ref{p:neg}) are strict except in the trivial case $\boldsymbol{\zeta} = (0, 0, \ldots, 0)$.
\end{enumerate}
\end{definition}

The first hints of the negative type conditions may be traced back to an 1841 paper of Cayley \cite{Cay}. Notably,
a metric $d$ on a finite set $X$ is of $1$-negative type if and only if $(X, \sqrt{d})$ may be isometrically
embedded into some Euclidean space. This is a well-known consequence of Theorem 1 in Schoenberg \cite{Sc1}.
Faver \textit{et al}.\ \cite{Fav} modified Schoenberg's proof to show that a metric $d$ on a finite set $X$
is of strict $1$-negative type if and only if $(X, \sqrt{d})$ may be isometrically embedded into some Euclidean
space as an affinely independent set.
By way of application it follows that every finite simple connected graph endowed
with the ordinary graph metric may be \textit{realized} as an affinely independent set in some Euclidean
space (\cite[Corollary 3.8]{Fav}). We recall that graph realization only requires the length of each edge
to be preserved by the embedding. Other distances within the graph may be distorted. In more general
settings one may simply wish to embed a finite metric space of (strict) $p$-negative type into a normed space
such as $L_{1}$ or $L_{2}$ with minimal distortion. Being able to do so in the case $p = 1$ has significant
implications for the design of approximation algorithms \cite{Ar1, Cha, Ar2}. Lately, ultrametric spaces
(or, more specifically, $k$-hierarchically well-separated trees) have figured prominently in work on embeddings
of finite metric spaces. Interesting papers along these lines include Bartal \textit{et al}.\ \cite{Ba1, Ba2}.

In this paper we examine strict $p$-negative type properties of finite ultrametric spaces in the limit as $p \rightarrow
\infty$. Importantly, a metric space is ultrametric if and only if it has strict $p$-negative type for all
$p \geq 0$ (\cite[Corollary 5.3]{Fav}). Doust and Weston \cite{Do1} introduced a way to quantify the degree of
strictness of the inequalities (\ref{p:neg}). This notion underpins the present work.

\begin{definition}\label{enh:type}
Let $(X,d)$ be a metric space that has $p$-negative type for some $p \geq 0$. Then the
\textit{$p$-negative type gap of $(X,d)$} is defined to be the largest non-negative constant $\Gamma = \Gamma_{X}(p)$
that satisfies
\begin{eqnarray*}
\frac{\Gamma}{2} \left( \sum\limits_{k=1}^{n} |\zeta_{k}| \right)^{2}
+ \sum\limits_{j,i =1}^{n} d(z_{j},z_{i})^{p} \zeta_{j} \zeta_{i} & \leq & 0
\end{eqnarray*}
for all finite subsets $\{ z_{1}, \ldots, z_{n} \} \subseteq X$
and all choices of real numbers $\zeta_{1}, \ldots, \zeta_{n}$ with $\zeta_{1} + \cdots + \zeta_{n} = 0$.
\end{definition}

In practice, it is a non-trivial exercise in combinatorial optimization to determine the precise value of $\Gamma_{X}(p)$.
This is the case even if $|X|$ is relatively small. (A specific example is discussed in Remark \ref{ian:ex}.)
Most known results deal with what are seemingly the two most tractable cases: $p = 0$ or $p = 1$. A formula for the
$1$-negative type gap of a finite metric tree was derived in \cite{Do1} using Lagrange's multiplier theorem.
The same approach was used by Weston \cite{We2} to compute $\Gamma_{X}(0)$ for each finite metric space $(X,d)$.
Notably, the formula given for $\Gamma_{X}(0)$ in \cite{We2} only depends upon $|X|$. For a finite metric space $(X,d)$
of strict $p$-negative type, Wolf \cite{Wol} has derived some general matrix-based formulas for computing $\Gamma_{X}(p)$.
An application in \cite{Wol} computes the $1$-negative type gap of each odd cycle $C_{2k+1}$, $k \geq 1$.
The formulas in \cite{Wol} also require non-trivial combinatorial optimization and depend upon being able to find
the inverse of the $p$-distance matrix of the underlying finite metric space $(X,d)$. For any given $p \geq 0$,
Li and Weston \cite{Hli} have shown that a finite metric metric space $(X,d)$ has strict $p$-negative type if and only
if $\Gamma_{X}(p) > 0$. In particular, if $(X,d)$ is ultrametric, then $\Gamma_{X}(p) > 0$ for all $p \geq 0$.
Conversely, if $\Gamma_{X}(p) > 0$ for all $p \geq 0$, then $(X,d)$ is ultrametric. These statements follow from
aforementioned results in \cite{Fav} and \cite{Hli}. At this point it is worth recalling the formal definition of an ultrametric.

\begin{definition}\label{ultra}
A metric $d$ on a set $X$ is said to be \textit{ultrametric} if for all $x,y,z \in X$, we have:
\begin{eqnarray*}
d(x,y) & \leq & \max \{ d(x,z), d(y,z) \}.
\end{eqnarray*}
\end{definition}

De Groot \cite{deG} characterized ultrametric spaces up to homeomorphism as the strongly
zero-dimensional metric spaces. In fact, ultrametric spaces arise in very specific ways as the end spaces
or leaves of certain tree-like structures, and it is for this reason that they are exceptionally
important in fields as diverse as computational logic \cite{Mur}, data analysis \cite{Car}, non-commutative geometry
\cite{Hu2}, and quantum mechanics \cite{Koz}. Hughes \cite{Hu1} has given a categorical equivalence
between the end spaces of infinite trees and complete ultrametric spaces. More precisely;
\cite[Theorem 6.9]{Hu1} states that there is an equivalence from the category of geodesically complete,
rooted $\mathbb{R}$-trees and equivalence classes of isometries at infinity, to the category of
complete ultrametric spaces of finite diameter and local similarity equivalences. 

An important class of ultrametric spaces arise as dendrograms. Loosely speaking, a dendrogram is a nested family
of partitions of a set that is usually represented graphically as a rooted tree. Dendrograms appear in many applications,
particularly in the data mining technique of hierarchical clustering. Dendrograms come in two main flavors:
\textit{proximity} and \textit{threshold}. Proximity dendrograms and finite ultrametric spaces are equivalent.
Indeed, given a finite set $X$ with at least two points, there is a natural bijection between the collection of all
proximity dendrograms over $X$ and the collection of all ultrametrics on $X$. A nice account of this well-known
equivalence (which will be central in our arguments below) is given by Carlsson and M\'{e}moli \cite[Theorem 9]{Car}.

Now suppose that $(X,d)$ is a finite ultrametric space with minimum non-zero distance $\alpha$.
As noted, $\Gamma_{X}(p) > 0$ for all $p \geq 0$. In this paper we examine the limiting
behavior of the ratio $\Gamma_{X}(p) / \alpha^{p}$ as $p \rightarrow \infty$. It turns out that
there is an intriguing pattern. Theorem \ref{main} shows that the limit
\begin{eqnarray*}
\Gamma_{X}(\infty) & \defeq & \lim\limits_{p \rightarrow \infty} \frac{\Gamma_{X}(p)}{\alpha^{p}}
\end{eqnarray*}
exists and is necessarily finite. Indeed, we derive an explicit combinatorial formula for $\Gamma_{X}(\infty)$.
Before stating this formula we need to introduce some additional notation and terminology. For each integer
$n \geq 2$, we set
\begin{eqnarray}\label{theta}
\vartheta (n) & \defeq & \frac{1}{2} \left( \left\lfloor \frac{n}{2} \right\rfloor^{-1}
+ \left\lceil \frac{n}{2} \right\rceil^{-1} \right).
\end{eqnarray}

The next notion has a simple interpretation in terms of proximity dendrograms. (See Remark \ref{cot:rem}.)

\begin{definition}\label{cots}
Let $(X,d)$ be a finite ultrametric space with minimum non-zero distance $\alpha$. Let $z \in X$ be given.
The closed ball $B_{z}(\alpha) = \{ x \in X : d(x,z) \leq \alpha \}$ will be called a \textit{coterie in} $X$
if $|B_{z}(\alpha)| > 1$.
\end{definition}

It is a simple matter to verify the following statements about coteries. If $(X,d)$ is a finite ultrametric space
with at least two points, then $X$ contains at least one coterie. Moreover, if $B_{1}$ and $B_{2}$ are coteries
in $X$, then $B_{1} = B_{2}$ or $B_{1} \cap B_{2} = \emptyset$. We are now in a position to formulate
the statements of our main results.

\begin{theorem}\label{main}
Let $(X,d)$ be a finite ultrametric space with at least two points and minimum non-zero distance $\alpha$.
If $B_{1}, B_{2}, \ldots, B_{l}$ are the distinct coteries of $(X,d)$, then
\begin{eqnarray}\label{limit}
\lim\limits_{p \rightarrow \infty} \frac{\Gamma_{X}(p)}{\alpha^{p}} & = &
\left\{ \vartheta (|B_{1}|)^{-1} + \cdots + \vartheta (|B_{l}|)^{-1} \right\}^{-1}.
\end{eqnarray}
\end{theorem}

As noted, the left side of (\ref{limit}) will be denoted $\Gamma_{X}(\infty)$. We call $\Gamma_{X}(\infty)$
the \textit{asymptotic negative type constant of $(X,d)$}. In proving Theorem \ref{main}, we will
see that the ratio $\Gamma_{X}(p)/ \alpha^{p}$ is non-decreasing on $[0, \infty)$.
(We note in Remark \ref{rem:1} that no such statement holds for finite metric spaces that are not ultrametric.)
Thus, by definition of $\Gamma_{X}(p)$, we obtain the following automatic corollary of Theorem \ref{main}.

\begin{corollary}\label{main:cor}
Let $(X,d)$ be a finite ultrametric space with at least two points and minimum non-zero distance $\alpha$.
Let $G \in (0, \Gamma_{X}(\infty))$ be given. Then, for all sufficiently large $p$, we have
\begin{eqnarray*}
\frac{G \cdot \alpha^{p}}{2} \left( \sum\limits_{k=1}^{n} |\zeta_{k}| \right)^{2}
+ \sum\limits_{j,i =1}^{n} d(z_{j},z_{i})^{p} \zeta_{j} \zeta_{i} & \leq & 0
\end{eqnarray*}
for each finite subset $\{ z_{1}, \ldots, z_{n} \} \subseteq X$
and each choice of real numbers $\zeta_{1}, \ldots, \zeta_{n}$ with $\zeta_{1} + \cdots + \zeta_{n} = 0$.
\end{corollary}

On the basis of Theorem \ref{main} and existing results in the literature, it is also possible to
determine when $\Gamma_{X}(p)/ \alpha^{p}$ is constant on $[0, \infty)$. The trivial case is the discrete metric.
If the ultrametric $d$ on the set $X$ is given by $d = \alpha \cdot \rho$ where $\rho$ is the discrete metric on $X$,
then it follows from Weston \cite[Theorem 3.2]{We2} (and (\ref{Y}) below) that $\Gamma_{X}(p)/ \alpha^{p}$ is
constant on $[0, \infty)$. We complete the classification of all such ultrametrics by deriving the following theorem.

\begin{theorem}\label{main2}
Let $(X,d)$ be a finite ultrametric space with at least two points and minimum non-zero distance $\alpha$.
Assume that $d \not= \alpha \cdot \rho$ where $\rho$ denotes the discrete metric on $X$.
If $B_{1}, B_{2}, \ldots, B_{l}$ are the distinct coteries of $(X,d)$, then the following conditions are equivalent:
\begin{enumerate}
\item $\Gamma_{X}(0) = \Gamma_{X}(\infty)$

\item $\Gamma_{X}(p)/ \alpha^{p}$ is constant on $[0, \infty)$

\item $|X| = |B_{1}| + \cdots |B_{l}|$ and all of the integers $|B_{1}|, \ldots, |B_{l}|$ are even.
\end{enumerate}
\end{theorem}

A number of basic conventions are used in this paper. $\mathbb{N}$ denotes the set of positive integers and
is referred to as the set of natural numbers. Given $n \in \mathbb{N}$, we use $[n]$ to denote the
segment of the first $n$ natural numbers: $\{ 1, 2, \ldots, n \}$. Sums indexed over the empty set are always
taken to be $0$. In relation to Definition \ref{neg:type} in the case $p = 0$ (and elsewhere), we define $0^{0} = 0$.

The rest of the paper is organized as follows. Section \ref{sec:sim} discusses
normalized simplices and a technique for computing $\Gamma_{X}(p)$.
Section \ref{sec:den} recalls the natural bijection between proximity dendrograms and finite ultrametric spaces.
Section \ref{sec:4} develops a combinatorial framework to underpin our proof of Theorem \ref{main}.
Section \ref{sec:5} then uses analytical techniques to complete the derivation of Theorem \ref{main}.
In Section \ref{sec:6} we conclude the paper by proving Theorem \ref{main2}.

\section{Normalized simplices and the calculation of $\Gamma_{X}(p)$}\label{sec:sim}

In this section we recall a method for computing the $p$-negative type gap $\Gamma_{X}(p)$ of a metric space $(X,d)$.
The technique is based upon the notion of normalized simplices.

\begin{definition}\label{S1}
Let $X$ be a set and suppose that $s,t > 0$ are integers.
A \textit{normalized $(s,t)$-simplex in $X$} is a collection of pairwise distinct points
$x_{1}, \ldots, x_{s}, y_{1}, \ldots, y_{t} \in X$ together with a corresponding collection
of positive real numbers $m_{1}, \ldots, m_{s}, n_{1}, \ldots, n_{t}$ that satisfy $m_{1} + \cdots + m_{s}
= 1 = n_{1} + \cdots + n_{t}$. Such a configuration of points and real numbers will be denoted by
$D = [x_{j}(m_{j});y_{i}(n_{i})]_{s,t}$ and will simply be called a \textit{simplex} when
no confusion can arise. We will call the vertices $x_{1}, \ldots, x_{s} \in D$ the \textit{$M$-team} and the
vertices $y_{1}, \ldots, y_{t} \in D$ the \textit{$N$-team}.
\end{definition}

Simplices with weights on the vertices were introduced by Weston \cite{We1} to study the generalized
roundness of finite metric spaces. The basis for the following definition is derived from the original
formulation of generalized roundness due to Enflo \cite{Enf}. Enflo introduced generalized roundness
in order to address a problem of Smirnov concerning the uniform structure of Hilbert spaces. There are
intimate ties between negative type and generalized roundness. See, for example, Prassidis and Weston
\cite{Pra}.

\begin{definition}\label{S4}
Suppose $p \geq 0$ and let $(X,d)$ be a metric space.
For each simplex $D = [x_{j}(m_{j});y_{i}(n_{i})]_{s,t}$ in $X$
we define
\begin{eqnarray*}
\gamma_{p}(D) & \defeq &
\sum\limits_{j,i = 1}^{s,t} m_{j}n_{i}d(x_{j},y_{i})^{p} -
\sum\limits_{ 1 \leq j_{1} < j_{2} \leq s} m_{j_{1}}m_{j_{2}}d(x_{j_{1}},x_{j_{2}})^{p} -
\sum\limits_{ 1 \leq i_{1} < i_{2} \leq t} n_{i_{1}}n_{i_{2}}d(y_{i_{1}},y_{i_{2}})^{p}.
\end{eqnarray*}
We call $\gamma_{p}(D)$ the \textit{$p$-simplex gap of $D$ in $(X,d)$}.
\end{definition}

It is worth noting that for any given $p \geq 0$, simplices and simplex gaps may be used to characterize metric spaces of
(strict) $p$-negative type. Indeed, Lennard \textit{et al}.\ \cite{Ltw} established
that a metric space $(X,d)$ has $p$-negative type if and only if $\gamma_{p}(D) \geq 0$
for each simplex $D$ in $X$. Doust and Weston \cite{Do1} further noted that a metric space $(X,d)$
has strict $p$-negative type if and only if $\gamma_{p}(D) > 0$ for each simplex $D$ in $X$.
The main significance of simplices and $p$-simplex gaps for this paper is the following theorem.

\begin{theorem}\label{S5}
Let $(X,d)$ be a metric space with $p$-negative type for some $p \geq 0$. Then $\Gamma_{X}(p) = \inf \gamma_{p}(D)$
where the infimum is taken over all simplices $D = [x_{j}(m_{j});y_{i}(n_{i})]_{s,t}$ in $X$.
\end{theorem}

The proof of Theorem \ref{S5} in the case $p = 1$ is due to Doust and Weston \cite[Theorem 4.16]{Do2}.
The proof in \cite{Do2} may be easily adapted to deal with any fixed $p \geq 0$.

\begin{remark}\label{compact}
The set of simplices in a finite set $X$ can
be identified with a compact metric space $\mathcal{N}$. The first step is to fix an enumeration
of the elements of $X$. Say, $X= \{ z_{1}, z_{2}, \ldots, z_{n} \}$. Let
\begin{eqnarray*}
\mathcal{N} & \defeq & \left\{ \boldsymbol{\omega}
= (\omega_{k}) \in \mathbb{R}^{n} : \sum \omega_{k} = 0 \text{ and } \sum |\omega_{k}| = 2 \right\}.
\end{eqnarray*}
The set $\mathcal{N}$ inherits the Euclidean metric on $\mathbb{R}^{n}$ and is thus a compact metric space.
For each $\boldsymbol{\omega} \in \mathcal{N}$ we may construct a simplex in the
following manner. If $\omega_{j} > 0$, we put $z_{j}$ on the $M$-team with corresponding weight
$\omega_{j}$. If $\omega_{i} < 0$, we put $z_{i}$ on the $N$-team with corresponding weight $- \omega_{i}$.
By relabeling, the $M$ and $N$-teams may be enumerated as $x_{1}, \ldots, x_{s}$ and $y_{1}, \ldots, y_{t}$
with corresponding weights $m_{1}, \ldots, m_{s}$ and $n_{1}, \ldots, n_{t}$, respectively.
(More precisely, define $x_{1} = z_{j}$ where $j$ is the smallest $k \in [n]$ such that $\omega_{k} > 0$, and so on.)
The simplex $D = [x_{j}(m_{j});y_{i}(n_{i})]_{s,t}$ generated in this way will be denoted $\boldsymbol{\omega}$,
allowing one to write $\gamma_{p}(\boldsymbol{\omega})$ instead of $\gamma_{p}(D)$.
It is plain that every simplex in $X$ arises in the way we have described. Moreover, given $p \geq 0$ and
$\boldsymbol{\omega} \in \mathcal{N}$, it is not difficult to verify that
\begin{eqnarray}\label{continuity}
2 \cdot \gamma_{p}(\boldsymbol{\omega}) & = & - \boldsymbol{\omega} \cdot D_{p} \boldsymbol{\omega},
\end{eqnarray}
where $D_{p}$ denotes the so-called $p$-\textit{distance matrix} $(d(z_{i}, z_{j})^{p})_{n,n}$.
\end{remark}

\begin{theorem}\label{gap:evl}
Let $(X,d)$ be a finite metric space with $p$-negative type for some $p \geq 0$. Then there exists
a simplex $\boldsymbol{\omega}_{0}$ in $X$ such that $\Gamma_{X}(p) = \gamma_{p}(\boldsymbol{\omega}_{0})$.
\end{theorem}

\begin{proof}
Assuming an enumeration of $X$ as in Remark \ref{compact} we let $\mathcal{N}$ denote the set
of simplices in $X$. For each $p \geq 0$ it is immediately evident from (\ref{continuity}) that
the function $\gamma_{p} : \mathcal{N} \rightarrow [0, \infty) : \boldsymbol{\omega}
\rightarrow \gamma_{p}(\boldsymbol{\omega})$ is continuous. As $\mathcal{N}$ is a compact
metric space it then follows that the infimum
\begin{eqnarray*}
\Gamma_{X}(p) & = & \inf\limits_{\boldsymbol{\omega} \in \mathcal{N}} \gamma_{p}(\boldsymbol{\omega})
\end{eqnarray*}
is attained for some simplex $\boldsymbol{\omega}_{0} \in \mathcal{N}$.
\end{proof}

\section{Proximity dendrograms and finite ultrametric spaces}\label{sec:den}

In this section we briefly recall the bijection between proximity dendrograms and finite ultrametric
spaces. Our approach and notation is largely based on that of Carlsson and M\'{e}moli \cite{Car}.

\begin{definition}
Given a non-empty finite set $X$, we let $\mathcal{P}(X)$ denote the set of all partitions of $X$.
Given a partition $\pi \in \mathcal{P}(X)$ we call each element $\mathbf{v} \in \pi$ a \textit{block}
of $\pi$. Given partitions $\pi_{1}, \pi_{2} \in \mathcal{P}(X)$ we say that $\pi_{1}$
is a \textit{refinement} of $\pi_{2}$ if each block of $\pi_{1}$ is contained in some block of $\pi_{2}$.
A refinement $\pi_{1}$ of $\pi_{2}$ is said to be \textit{proper} if $\pi_{1} \not= \pi_{2}$.
\end{definition}

Proximity dendrograms are defined in terms of certain nested collections of partitions as follows.

\begin{definition}\label{prox:d}
A \textit{proximity dendrogram} is a triple $\mathcal{D} = (X, \{ \alpha_{0}, \alpha_{1}, \ldots, \alpha_{\ell} \}, \pi)$
consisting of a non-empty finite set $X$, finitely many real numbers $0 = \alpha_{0} < \alpha_{1} < \cdots < \alpha_{\ell}$
and a function $\pi : \{ \alpha_{0}, \alpha_{1}, \ldots, \alpha_{\ell} \} \rightarrow \mathcal{P}(X)$ with the
following properties:
\begin{enumerate}
\item $\pi(\alpha_{0}) = \{ \{ x \} : x \in X \}$,

\item $\pi(\alpha_{\ell}) = \{ X \}$, and

\item if $k \in [\ell]$, then $\pi(\alpha_{k-1})$ is a proper refinement of $\pi(\alpha_{k})$.
\end{enumerate}
We call the set $\{ \alpha_{0}, \alpha_{1}, \ldots, \alpha_{\ell} \}$ the \textit{proximity part} of $\mathcal{D}$.
\end{definition}

Notice that conditions (1), (2) and (3) in Definition \ref{prox:d} imply that $\ell \geq 1$. Every proximity
dendrogram generates a directed rooted tree $T$ and this ultimately leads to the definition of an ultrametric on $X$.

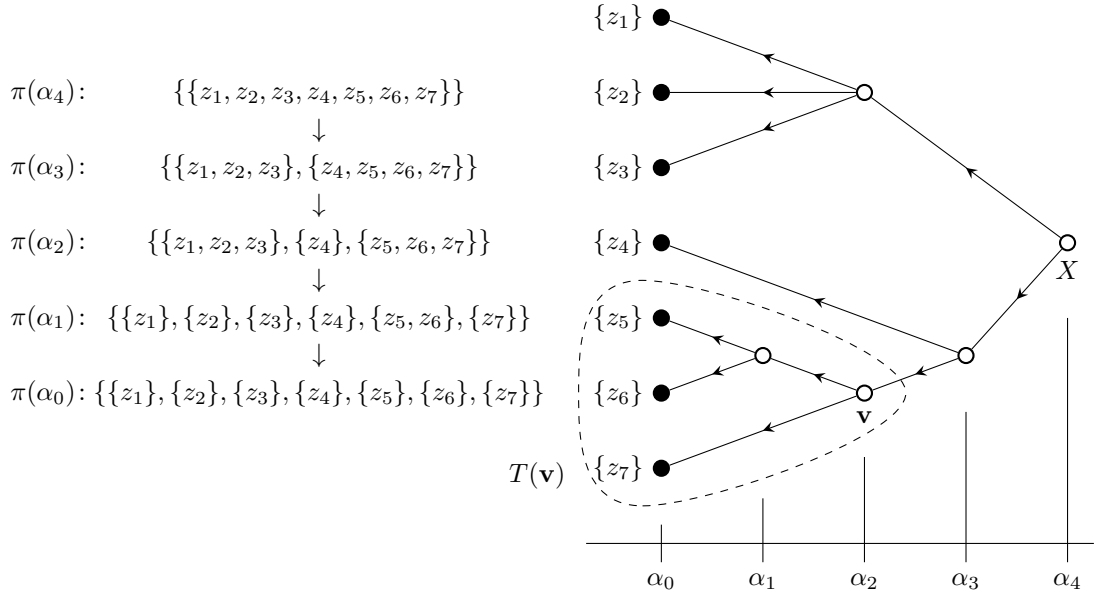
\begin{figure}[h!]
\begin{center}
\begin{tikzpicture}[scale=0.5] 

 \node (p4) at (-16.3, 10) {$\pi(\alpha_4)\!:$};
 \path (p4) + (7.2,0) node (4l) { $\{ \{ z_{1}, z_{2}, z_{3}, z_{4}, z_{5}, z_{6}, z_{7}\} \}$ };
 
 \path (4l) + (0,-1) node (4a) {$\downarrow$};
 
 \path (p4) + (0,-2) node (p3) {$\pi(\alpha_3)\!:$};
 \path (4l) + (0,-2) node (3l)  {  $\{ \{ z_{1}, z_{2}, z_{3}\},  \{ z_{4}, z_{5}, z_{6}, z_{7}\} \}$  };
 
 \path (3l) + (0,-1) node (3a) {$\downarrow$};
 
 \path (p3) + (0,-2) node (p2) {$\pi(\alpha_2)\!:$};
 \path (3l) + (0,-2) node (2l) { $\{ \{ z_{1}, z_{2}, z_{3}\}, \{ z_{4}\}, \{ z_{5}, z_{6}, z_{7}\} \}$ };
  
 \path (2l) + (0,-1) node (2a)  {$\downarrow$};
  
 \path (p2) + (0,-2) node (p1) {$\pi(\alpha_1)\!:$}; 
 \path (2l) + (0,-2) node (1l) { $\{ \{ z_{1}\} ,\{ z_{2}\} ,\{ z_{3}\}, \{ z_{4}\}, \{ z_{5}, z_{6}\}, \{ z_{7} \} \}$ };
 
 \path (1l) + (0,-1) node (1a) {$\downarrow$};
 
 \path (p1) + (0,-2) node (p0) {$\pi(\alpha_0)\!:$};
 \path (1l) + (0,-2) node (0l)  { $\{ \{ z_{1}\} ,\{ z_{2}\} ,\{ z_{3}\}, \{ z_{4}\}, \{ z_{5}\}, \{ z_{6}\}, \{ z_{7}\} \}$ };

 \node[black,label=left:{$\{ z_{7}\}$}] (1) at (0,0) {};
 \node[black,label=left:{$\{ z_{6}\}$}] (2) at (0,2) {};
 \node[black,label=left:{$\{ z_{5}\}$}] (3) at (0,4) {};
 \node[black,label=left:{$\{ z_{4}\}$}] (4) at (0,6) {};
 \node[black,label=left:{$\{ z_{3}\}$}] (5) at (0,8) {};
 \node[black,label=left:{$\{ z_{2}\}$}] (6) at (0,10) {};
 \node[black,label=left:{$\{ z_{1}\}$}] (12) at (0,12) {};
 
 \node[white] (7) at (2.7,3) {};
 
 \node[white,label=below:{$\mathbf{v}$}] (8) at (5.4,2) {};
 \node[white] (9) at (5.4,10) {};
 
 \node[white] (10) at (8.1,3) {};
 
 \node[white,label=below:{$X$}] (11) at (10.8,6) {};
 
 \draw [->-,>=stealth]  (8) -- (1) ;
 \draw [->-,>=stealth] (7) -- (2) ;
 \draw [->-,>=stealth]  (7) -- (3) ;
 \draw [->-,>=stealth]  (8) -- (7) ;
 \draw [->-,>=stealth]  (10) -- (8) ;
 \draw [->-,>=stealth]  (10) -- (4) ;
 \draw [->-,>=stealth]  (9) -- (5) ;
 \draw [->-,>=stealth]  (9) -- (6) ;
 \draw [->-,>=stealth]  (9) -- (12) ;
 \draw [->-,>=stealth]  (11) -- (9) ;
 \draw [->-,>=stealth]  (11) -- (10) ;

 \draw (-2,-2) -- (11.5,-2) ; 
 
 \draw (0,-2.5) -- (0,-1.5) ;
 \node at (0,-3) {$\alpha_0$} ;
 
 \draw (2.7,-2.5) -- (2.7,-0.8);
 \node at (2.7,-3) {$\alpha_1$} ;
 
 \draw (5.4,-2.5) -- (5.4,0.3) ;
 \node at (5.4,-3) {$\alpha_2$};
 
 \draw (8.1,-2.5) -- (8.1,1.5) ;
 \node at (8.1,-3) {$\alpha_3$};
 
 \draw (10.8,-2.5) -- (10.8,4) ;
 \node at (10.8,-3) {$\alpha_4$};

 \draw[dashed] plot [smooth cycle,tension=0.7] coordinates {(-2.2,2)  (-0.5,-1)  (6.5,2) (-0.5,5)}; 
 
 \node at (-3.3,-0.2) {$T(\mathbf{v})$};
 
\end{tikzpicture} 
\end{center} 
\caption{A proximity dendrogram with the sequence of partitions on the left and the generated directed tree on the right.
The subtree $T(\mathbf{v})$ generated by the node $\mathbf{v}$ is circled.}
\end{figure}

\begin{definition}\label{prox:tree}
Let $\mathcal{D} = (X, \{ \alpha_{0}, \alpha_{1}, \ldots, \alpha_{\ell} \}, \pi)$ be a given proximity dendrogram.
\begin{enumerate}
\item[(a)] The nodes $\mathbf{v}$ of the \textit{tree $T = T(\mathcal{D})$ generated by $\mathcal{D}$} consist of the blocks
appearing in the partitions $\pi(\alpha_0),\dots,\pi(\alpha_{\ell})$.

\item[(b)] The edge set of $T$ is determined by
the Hasse diagram for this node set, partially ordered by set inclusion.

\item[(c)] If there is an edge between nodes $\mathbf{v}$
and $\mathbf{w}$ in $T$, it is directed to go from $\mathbf{v}$ to $\mathbf{w}$ if $\mathbf{w} \subset \mathbf{v}$.

\item[(d)] If there is a directed edge from $\mathbf{v}$ to $\mathbf{w}$ in $T$
we say that $\mathbf{w}$ is \textit{left-adjacent} to $\mathbf{v}$ in $T$. We let $\Adj (\mathbf{v})$
denote the set of nodes $\mathbf{w} \in T$ that are left-adjacent to $\mathbf{v}$.
\end{enumerate}
\end{definition}

More generally, a node $\mathbf{w}$ is said to be to the \textit{left} of a node $\mathbf{v}$ in
$T$ if $\mathbf{w} \subset \mathbf{v}$. Notice that the leaves of $T$ are the blocks in
$\pi(\alpha_{0})$ and the root is the block $\mathbf{r} = X$. Additional structure of $T$ is defined as follows.

\begin{definition} Let $\mathcal{D} = (X, \{ \alpha_{0}, \alpha_{1}, \ldots, \alpha_{\ell} \}, \pi)$
be a given proximity dendrogram. Let $T = T(\mathcal{D})$.
\begin{enumerate}
\item[(a)] The \textit{level} of a node $\mathbf{v} \in T$ is the smallest $k$ such that
$\mathbf{v} \in \pi(\alpha_{k})$. We let $\Pi_{k}$ denote the set of all level $k$ nodes in $T$,
$0 \leq k \leq \ell$.

\item[(b)] The \textit{subtree $T(\mathbf{v})$ of $T$ generated by a level $k$ node $\mathbf{v} \in \Pi_{k}$}
is the directed subtree of $T$ that consists of $\mathbf{v}$ together with all nodes and edges that are to the
left of $\mathbf{v}$.

\item[(c)] The \textit{left-degree} $b(\mathbf{v})$ of a node $\mathbf{v} \in T$ is the degree of the node
$\mathbf{v}$ in the subtree $T(\mathbf{v})$.
\end{enumerate}
\end{definition}

Notice that $b(\mathbf{v})$ signifies the number of directed edges in $T$ emanating to the left from
$\mathbf{v}$. Furthermore, if $k \in [\ell]$ and $\mathbf{v} \in \Pi_{k}$, then $b(\mathbf{v}) \geq 2$.
Clearly, $b(\mathbf{v}) = |\Adj (\mathbf{v})|$ for all $\mathbf{v} \in T$.

Proximity dendrograms give rise to finite ultrametric spaces. Indeed,
given a proximity dendrogram $\mathcal{D} = (X, \{ \alpha_{0}, \alpha_{1}, \ldots, \alpha_{\ell} \}, \pi)$
and given $x,y \in X$, it is not difficult to verify that
\begin{eqnarray*}
d_{\mathcal{D}}(x,y) & \defeq & \min \{ \alpha_{k} : x,y \text{ belong to the same block of } \pi(\alpha_{k}) \}
\end{eqnarray*}
defines an ultrametric on $X$. Conversely, if $d$ is an ultrametric on a finite set $X$ with at least two
elements, then there is a unique proximity dendrogram $\mathcal{D} = (X, \{ \alpha_{0}, \alpha_{1}, \ldots, \alpha_{\ell} \}, \pi)$
such that $d = d_{\mathcal{D}}$. (A detailed explanation of this bijection is given in Carlsson and M\'{e}moli
\cite[Theorem 9]{Car}.) We call $\mathcal{D}$ the \textit{proximity dendrogram associated with $(X,d)$}.
Notice that $\alpha_{1} < \alpha_{2} < \cdots < \alpha_{\ell}$ are the non-zero distances in $(X,d)$.

\begin{remark}\label{cot:rem}
Notice that if $\mathcal{D} = (X, \{ \alpha_{0}, \alpha_{1}, \ldots, \alpha_{\ell} \}, \pi)$
is the proximity dendrogram associated with a finite ultrametric space $(X,d)$ then, by Definition \ref{cots}, the
coteries in $X$ are precisely the blocks in $\Pi_{1}$.
\end{remark}

Consider a given finite ultrametric space $(X,d)$, $|X| > 1$. Suppose that
$\mathcal{D} = (X, \{ \alpha_{0}, \alpha_{1}, \ldots, \alpha_{\ell} \}, \pi)$ is the
proximity dendrogram associated with $(X,d)$. Let $p > 0$ be given. Clearly $d^{p}$
is also an ultrametric on $X$. Moreover, the proximity dendrogram associated with
$(X, d^{p})$ is simply $\mathcal{D}_{p} \defeq
(X, \{ \alpha_{0}^{p}, \alpha_{1}^{p}, \ldots, \alpha_{\ell}^{p} \}, \pi^{\prime})$, where
$\pi^{\prime}(\alpha_{k}^{p}) = \pi(\alpha_{k})$ for each $k$, $0 \leq k \leq \ell$.
It follows that the trees generated by $\mathcal{D}$ and $\mathcal{D}_{p}$ are identical.
Thus $T(\mathcal{D}) = T(\mathcal{D}_{p})$ for all $p > 0$.
Likewise, $d^{\prime} \defeq d / \alpha_{1}$ is an ultrametric on $X$. Moreover,
$(X,d)$ and $(X,d^{\prime})$ share the same proximity dendrogram tree $T$. The minimum non-zero distance
in $(X,d^{\prime})$ is $1$ and it follows from Definition \ref{enh:type} that
\begin{eqnarray}\label{Y}
\Gamma_{(X,d)}(p) & = & \alpha_{1}^{p} \cdot \Gamma_{(X, d^{\prime})}(p).
\end{eqnarray}
Provided this scaling is ultimately taken into account we may, henceforth, restrict our attention to
those finite ultrametric spaces that have minimum non-zero distance $\alpha_{1} = 1$. This assumption
greatly simplifies the following calculations and it must therefore be regarded as being more than a
purely cosmetic maneouver.

\section{Combinatorial properties of ultrametric simplex gaps $\gamma_{p}(\boldsymbol{\omega})$}\label{sec:4}

Let $d$ be a given ultrametric on a finite set $X = \{ z_{1}, \ldots, z_{n} \}$ where $n > 1$. Assume
that the minimum non-zero distance in $(X,d)$ is $\alpha_{1} = 1$. We may identify the set of all simplices
in $X$ with the compact metric space $\mathcal{N} \subset \mathbb{R}^{n}$ described in Remark \ref{compact}.
Let $\mathcal{D} = (X, \{ \alpha_{0}, \alpha_{1}, \ldots, \alpha_{\ell} \}, \pi)$ denote the unique
proximity dendrogram associated with $(X,d)$ and set $T = T(\mathcal{D})$. Lastly,
fix a simplex $\boldsymbol{\omega} = [x_{j}(m_{j}); y_{i}(n_{i})]_{s,t} \in \mathcal{N}$. We examine the
behavior of $\gamma_{p}(\boldsymbol{\omega})$ as $p \rightarrow \infty$.
The arguments are based on the following notions.

\begin{definition}
For each node $\mathbf{v} \in T$ we define the \textit{block sums}:
\begin{eqnarray*}
M(\mathbf{v}) & \defeq & \sum\limits_{j : \{ x_{j} \} \in T(\mathbf{v})} m_{j}, \text{ and} \\
N(\mathbf{v}) & \defeq & \sum\limits_{i : \{ y_{i} \} \in T(\mathbf{v})} n_{i}.
\end{eqnarray*}
We say that the subtree $T(\mathbf{v})$ is \textit{simplicially balanced} if $M(\mathbf{v}) = N(\mathbf{v})$.
\end{definition}

\begin{remark}
Notice that if $\mathbf{v}$ is a block of the form $\{ x_{j} \}$ or $\{ y_{i} \}$, then
$|M(\mathbf{v}) - N(\mathbf{v})| = m_{j}$ or $n_{i}$ (respectively), and so $T(\mathbf{v})$ is not simplicially balanced.
\end{remark}

Suppose that $p > 0$ and let $\gamma(p) \defeq \gamma_{p}(\boldsymbol{\omega})$. Note that $\gamma(p) > 0$ because
$(X,d)$ has strict $p$-negative type. We begin by developing a formula for $\gamma(p)$ in terms of the block sums
defined above. By Definition \ref{S4},
\begin{eqnarray*}
\gamma(p) & = & \sum\limits_{j,i = 1}^{s,t} m_{j}n_{i}d^{p}(x_{j},y_{i}) \\
& ~ & - \sum\limits_{1 \leq j_{1} < j_{2} \leq s} m_{j_{1}}m_{j_{2}}d^{p}(x_{j_{1}},x_{j_{2}})
- \sum\limits_{1 \leq i_{1} < i_{2} \leq t} n_{i_{1}}n_{i_{2}}d^{p}(y_{i_{1}},y_{i_{2}}).
\end{eqnarray*}
Since the non-zero distances in $(X,d)$ are $\alpha_{1}, \ldots, \alpha_{\ell}$ we see that
$\gamma(p) = c_{1}\alpha_{1}^{p} + \cdots + c_{\ell}\alpha_{\ell}^{p}$, where $c_{1}, \ldots, c_{\ell}$
are constants to be determined.

\subsection{Evaluation and properties of the constants $c_{k}$}

Suppose $k \in [\ell]$. We wish to determine a formula for $c_{k}$. We begin by noting that for all
$x,y \in X$, $d(x, y) = \alpha_{k}$ if and only there exists a level $k$ node $\mathbf{v} \in T$ such that
$x$ and $y$ belong to distinct blocks of $\Adj (\mathbf{v})$. Hence, using Definition \ref{S4},
\begin{eqnarray*}
c_{k} = \sum\limits_{\mathbf{v} \in \Pi_{k}} c_{k}(\mathbf{v}),
\end{eqnarray*}
where $c_{k}(\mathbf{v})$ is defined as follows.

\begin{definition}
Let $\mathbf{v} \in \Pi_{k}$. Then,
\begin{eqnarray*}
c_{k}(\mathbf{v}) & \defeq &
\sum\limits_{\stackrel{i,j = 1}{i \not= j}}^{b(\mathbf{v})}  M(\mathbf{u}_{i})N(\mathbf{u}_{j})
- \sum\limits_{1 \leq i < j \leq b(\mathbf{v})}  \left\{ M(\mathbf{u}_{i})M(\mathbf{u}_{j}) +
N(\mathbf{u}_{i})N(\mathbf{u}_{j}) \right\},
\end{eqnarray*}
where $\mathbf{u}_{1}, \ldots, \mathbf{u}_{b(\mathbf{v})}$ denote the
distinct blocks of $\Adj (\mathbf{v})$.
\end{definition}

Lemma \ref{lem:1} provides a formula for each $c_{k}(\mathbf{v})$ in terms of certain block sums.
The proof is expedited by the following easily verified identity.

\begin{lemma}\label{lem:0}
For any integer $k > 1$ and any choice of real numbers
$a_{1}, \ldots, a_{k}, b_{1}, \ldots, b_{k}$, we have
\begin{eqnarray*}
\sum\limits_{\stackrel{i,j = 1}{i \not= j}}^{k}  a_{i}b_{j}
- \sum\limits_{i < j}  \left\{  a_{i}a_{j} + b_{i}b_{j} \right\} & = &
\sum\limits_{i=1}^{k} \frac{(a_{i} - b_{i})}{2} \cdot
\left\{ \sum\limits_{\stackrel{j = 1}{j \not= i}}^{k} (b_{j} - a_{j}) \right\}.
\end{eqnarray*}
\end{lemma}

\begin{lemma}\label{lem:1}
For each $k \in [\ell]$ and each level $k$ node $\mathbf{v} \in T$,
\begin{eqnarray}\label{one}
2 c_{k}(\mathbf{v}) & = &
\left\{ \sum\limits_{\mathbf{u} \in \Adj (\mathbf{v})} (M(\mathbf{u}) - N(\mathbf{u}))^{2} \right\} -
(M(\mathbf{v}) - N(\mathbf{v}))^{2}.
\end{eqnarray}
If, moreover, the subtree $T(\mathbf{v})$ is simplicially balanced, then
$2 c_{k}(\mathbf{v}) = \sum\limits_{\mathbf{u} \in \Adj (\mathbf{v})} \left\{ (M(\mathbf{u}) - N(\mathbf{u}))^{2} \right\}$.
\end{lemma}

\begin{proof} We set $b = b(\mathbf{v})$ and let $\mathbf{u}_{1}, \ldots, \mathbf{u}_{b}$ denote the
distinct blocks of $\Adj (\mathbf{v})$. Notice then that
$M(\mathbf{v}) = \sum\limits_{j=1}^{b} M(\mathbf{u}_{j})$ and
$N(\mathbf{v}) = \sum\limits_{j=1}^{b} N(\mathbf{u}_{j})$. As a result we have
\begin{eqnarray*}
\sum\limits_{\stackrel{j = 1}{j \not= i}}^{b} (N(\mathbf{u}_{j}) - M(\mathbf{u}_{j}))
& = & (M(\mathbf{u}_{i}) - N(\mathbf{u}_{i})) - (M(\mathbf{v}) - N(\mathbf{v}))
\end{eqnarray*}
for each $i \in [b]$. By definition of $c_{k}(\mathbf{v})$ and Lemma \ref{lem:0}
we see that
\begin{eqnarray*}
2c_{k}(\mathbf{v}) & = &
\sum\limits_{\stackrel{i,j = 1}{i \not= j}}^{b} 2 M(\mathbf{u}_{i})N(\mathbf{u}_{j})
- \sum\limits_{1 \leq i < j \leq b} 2 \left\{ M(\mathbf{u}_{i})M(\mathbf{u}_{j}) +
N(\mathbf{u}_{i})N(\mathbf{u}_{j}) \right\} \\
& = & \sum\limits_{i=1}^{b} (M(\mathbf{u}_{i}) - N(\mathbf{u}_{i})) \cdot
\left\{ \sum\limits_{\stackrel{j = 1}{j \not= i}}^{b} (N(\mathbf{u}_{j}) - M(\mathbf{u}_{j})) \right\} \\
& = & \sum\limits_{i=1}^{b} (M(\mathbf{u}_{i}) - N(\mathbf{u}_{i})) \cdot
\left\{ (M(\mathbf{u}_{i}) - N(\mathbf{u}_{i})) - (M(\mathbf{v}) - N(\mathbf{v})) \right\} \\
& = & \sum\limits_{i=1}^{b} (M(\mathbf{u}_{i}) - N(\mathbf{u}_{i}))^{2} -
\sum\limits_{i=1}^{b} (M(\mathbf{u}_{i}) - N(\mathbf{u}_{i}))(M(\mathbf{v}) - N(\mathbf{v})) \\
& = & \sum\limits_{i=1}^{b} (M(\mathbf{u}_{i}) - N(\mathbf{u}_{i}))^{2} - (M(\mathbf{v}) - N(\mathbf{v}))^{2} \\
& = & \left\{ \sum\limits_{\mathbf{u} \in \Adj (\mathbf{v})} (M(\mathbf{u}) - N(\mathbf{u}))^{2} \right\} -
(M(\mathbf{v}) - N(\mathbf{v}))^{2}.
\end{eqnarray*}
If $T(\mathbf{v})$ is simplicially balanced, then (by definition) $M(\mathbf{v}) - N(\mathbf{v}) = 0$. The second assertion of the
lemma therefore follows immediately from the preceding calculation.
\end{proof}

Lemma \ref{lem:1} provides a means to determine when $\gamma(p)$ is constant on $(0, \infty)$.

\begin{theorem}\label{thm:3}
The following conditions are equivalent:
\begin{enumerate}
\item[(a)] $\gamma(p)$ is constant on $(0, \infty)$,

\item[(b)] $c_{\ell} = c_{\ell - 1} = \cdots = c_{2} = 0$,

\item[(c)] for each node $\mathbf{v} \in \Pi_{k}$ such that $k \geq 2$ and each $\mathbf{u} \in \Adj (\mathbf{v})$,
$T(\mathbf{u})$ is simplicially balanced, and

\item[(d)] $T(\mathbf{u})$ is simplicially balanced for each node $\mathbf{u} \in \Pi_{1}$ and no
$\{ x_{j} \}$ or $\{ y_{i} \}$ is adjacent to any level $k$ node with $k \geq 2$.
\end{enumerate}
Moreover, if $c_{\ell} = c_{\ell - 1} = \cdots = c_{r + 1} = 0$ for some $r \geq 2$, then
$c_{r}$ $(= c_{\ell} + c_{\ell - 1} + \cdots + c_{r})$ is positive or zero.
\end{theorem}

\begin{proof} Conditions (a) and (b) are plainly equivalent.

Suppose that condition (b) holds. The only level $\ell$ node is the root $\mathbf{r}$
of the entire dendrogram tree $T$. This entails that $c_{\ell} = c_{\ell}(\mathbf{r})$. Clearly $T = T(\mathbf{r})$ and
$T$ is simplicially balanced because $M(\mathbf{r}) = 1 = N(\mathbf{r})$ by definition of a simplex. Hence
\begin{eqnarray}\label{cl:0}
2 c_{\ell} & = & \sum\limits_{\mathbf{v} \in \Adj (\mathbf{r})} (M(\mathbf{v}) - N(\mathbf{v}))^{2}
\end{eqnarray}
by Lemma \ref{lem:1}. As we are assuming that $c_{\ell} = 0$ it follows immediately that $M(\mathbf{v}) = N(\mathbf{v})$
for each $\mathbf{v} \in \Adj(\mathbf{r})$. In other words, $T(\mathbf{v})$ is simplicially balanced for each
$\mathbf{v} \in \Adj(\mathbf{r})$. In particular, $T(\mathbf{v})$ is simplicially balanced for each level $\ell - 1$
node $\mathbf{v} \in T$ because all such nodes are adjacent to $\mathbf{r}$. So it follows from
Lemma \ref{lem:1} that
\begin{eqnarray}\label{cl1:0}
2 c_{\ell - 1}(\mathbf{v}) & = &
\sum\limits_{\mathbf{u} \in \Adj (\mathbf{v})} (M(\mathbf{u}) - N(\mathbf{u}))^{2} \\
& \geq & 0 \nonumber
\end{eqnarray}
for each level $\ell - 1$ node $\mathbf{v} \in T$. Now (by definition and assumption) $c_{\ell - 1}
= \sum c_{\ell - 1}(\mathbf{v}) = 0$. Thus $T(\mathbf{u})$ is simplicially balanced for all
$\mathbf{u} \in \Adj (\mathbf{v}), \mathbf{v} \in \Pi_{\ell - 1}$.
Now consider any integer $k \geq 2$ with the following property:
for any node $\mathbf{w}$ in any of the levels $\ell, \ell - 1, \ldots, k+1$
and any $\mathbf{u} \in \Adj(\mathbf{w})$ it is the case that $T(\mathbf{u})$ is simplicially balanced.
Then it must be the case $T(\mathbf{v})$ is simplicially balanced for each level $k$ node $\mathbf{v} \in T$. This is
because for any such node $\mathbf{v} \in T$ there is a unique integer $r > k$ such that $\mathbf{v} \in \Adj(\mathbf{w})$
for some level $r$ node $\mathbf{w} \in T$. It then follows from
Lemma \ref{lem:1} that
\begin{eqnarray*}
2 c_{k}(\mathbf{v}) & = &
\sum\limits_{\mathbf{u} \in \Adj (\mathbf{v})} (M(\mathbf{u}) - N(\mathbf{u}))^{2} \\
& \geq & 0 \nonumber
\end{eqnarray*}
for each level $k$ node $\mathbf{v} \in T$. Now (by definition and assumption) $c_{k}
= \sum c_{k}(\mathbf{v}) = 0$. Thus $T(\mathbf{u})$ is simplicially balanced for all
$\mathbf{u} \in \Adj (\mathbf{v}), \mathbf{v} \in \Pi_{k}$. We conclude that condition (c) holds.

Suppose that condition (c) holds. Condition (b) then follows from the cases $k = \ell, \ell - 1, \ldots, 2$ by
successive applications of Lemma \ref{lem:1}. For example, $c_{\ell} = c_{\ell}(\mathbf{r}) = 0$ by (\ref{cl:0}), then
$c_{\ell - 1} = \sum c_{\ell - 1}(\mathbf{v}) = 0$ by (\ref{cl1:0}), and so on. It is also clear that condition (c)
implies condition (d). This is because every level $1$ node is adjacent to some level $k$ node with $k \geq 2$
and, moreover, no $T(\{ x_{j} \})$ or $T(\{ y_{i} \})$ is simplicially balanced.

Suppose that condition (d) holds.
By way of illustration, consider an arbitrary level $2$ node $\mathbf{v} \in T$. Each node $\mathbf{u} \in \Adj (\mathbf{v})$
is a level $1$ node or a level $0$ node. If $\mathbf{u} \in \Adj (\mathbf{v})$ is a level $1$ node,
then $M(\mathbf{u}) = N(\mathbf{u})$ by hypothesis. If $\mathbf{u} \in \Adj (\mathbf{v})$ is a level $0$ node,
then $\mathbf{u}$ carries no weight (because no $\{ x_{j} \}$ or $\{ y_{i} \}$ is adjacent to any level $2$ node)
and so $M(\mathbf{u}) = 0 = N(\mathbf{u})$. Moreover,
\begin{eqnarray*}
M(\mathbf{v}) & = & \sum\limits_{\mathbf{u} \in \Adj (\mathbf{v}) \cap \Pi_{1}} M(\mathbf{u}) \\
              & = & \sum\limits_{\mathbf{u} \in \Adj (\mathbf{v}) \cap \Pi_{1}} N(\mathbf{u}) \\
              & = & N(\mathbf{v}).
\end{eqnarray*}
Thus $T(\mathbf{v})$ is simplicially balanced for every level $2$ node $\mathbf{v} \in T$. The argument now
proceeds in the obvious fashion to establish condition (c).

The final statement of the theorem also follows by successive applications of Lemma \ref{lem:1}. As $T = T(\mathbf{r})$
is simplicially balanced (because $M(\mathbf{r}) = 1 = N(\mathbf{r})$), Lemma \ref{lem:1} implies that $c_{\ell}$ is positive
or zero. If $c_{\ell} = 0$, then $c_{\ell - 1}(\mathbf{v}) \geq 0$ for each level $\ell - 1$ node $\mathbf{v} \in T$
by (\ref{cl1:0}), and so $c_{\ell - 1} = \sum c_{\ell - 1}(\mathbf{v})$ is positive or zero. If $c_{\ell} = c_{\ell - 1} = 0$,
then $T(\mathbf{v})$ will be simplicially balanced for each level $\ell - 2$ node $\mathbf{v} \in T$ (because any such
node is adjacent to a level $\ell - 1$ node or $\mathbf{r}$), and so $c_{\ell - 2} = \sum c_{\ell - 2}(\mathbf{v}) \geq 0$
by Lemma \ref{lem:1}. And so on.
\end{proof}

The final statement of Theorem \ref{thm:3} may be strengthened considerably.

\begin{theorem}\label{thm:4}
$\sum\limits_{i=k}^{\ell} c_{i} \geq 0$ for each $k \in [\ell]$.
Moreover, $2 (c_{1} + \cdots + c_{\ell}) = m_{1}^{2} + \cdots m_{s}^{2} + n_{1}^{2} + \cdots n_{t}^{2} > 0$.
\end{theorem}

\begin{proof}
As in (\ref{one}), consider a level $k$ node $\mathbf{v}$ but with the additional assumption that $k \not= \ell$.
Then there is a unique node $\mathbf{w} \in T$ such that $\mathbf{v} \in \Adj (\mathbf{w})$. Now $\mathbf{w}$
is a level $r$ node for some unique $r > k$. According to Lemma \ref{lem:1}, $-(M(\mathbf{v}) - N(\mathbf{v}))^{2}$ is the
lone negative term in the expression (\ref{one}) for $2c_{k}(\mathbf{v})$. On the other hand,
$+(M(\mathbf{v}) - N(\mathbf{v}))^{2}$ is a positive term in the corresponding expression for $2c_{r}(\mathbf{w})$.
So adding $2c_{k}(\mathbf{v})$ to $2c_{r}(\mathbf{w})$ will eliminate the terms $\pm (M(\mathbf{v}) - N(\mathbf{v}))^{2}$
by cancellation. As a result, due to repeated cancellations of this nature, it follows that
\begin{eqnarray*}
c_{k} + c_{k + 1} + \cdots + c_{\ell} & = & \sum \bigl\{ c_{r}(\mathbf{w}) : r \geq k \text{ and } \mathbf{w}
\text{ is a level } r \text{ node in } T \bigl\} \geq 0.
\end{eqnarray*}
Now suppose that $\mathbf{u} \in \Adj (\mathbf{v})$ where, again, $\mathbf{v}$ is a level $k$ node
in $T$ with $k \in [\ell]$. If $\mathbf{u}$ is an internal node of $T$ of level $q$ (say), then $q \in [k-1]$.
In this case, $+(M(\mathbf{u}) - N(\mathbf{u}))^{2}$ appears as a positive term in the expression (\ref{one})
for $2c_{k}(\mathbf{v})$ while $-(M(\mathbf{u}) - N(\mathbf{u}))^{2}$ is the lone negative term in the
corresponding expression for $2c_{q}(\mathbf{u})$. So adding $2c_{k}(\mathbf{v})$ to $2c_{q}(\mathbf{u})$
will once again eliminate the terms $\pm (M(\mathbf{u}) - N(\mathbf{u}))^{2}$ by cancellation. However, we cannot discount
the possibility that $\mathbf{u}$ may be a leaf of $T$. Indeed, if $\mathbf{u}$ is a level $0$ node in
$T$, then the positive term $+(M(\mathbf{u}) - N(\mathbf{u}))^{2}$ in the expression for
$2c_{k}(\mathbf{v})$ is not cancelled in the following sum:
\begin{eqnarray*}
c_{1} + c_{2} + \cdots + c_{\ell} & = & \sum \bigl\{ c_{j}(\mathbf{w}) : r \geq 1 \text{ and } \mathbf{w}
\text{ is a level } r \text{ node in } T \bigl\}.
\end{eqnarray*}
All other terms in the expression for $2c_{k}(\mathbf{v})$ are cancelled by our preceding comments. Thus:
\begin{eqnarray}\label{two}
2(c_{1} + c_{2} + \cdots + c_{\ell})
& = & \sum \bigl\{ (M(\mathbf{u}) - N(\mathbf{u}))^{2} : \mathbf{u} \text{ is a leaf of } T \bigl\} \nonumber \\
& = & \sum \bigl\{ (M(\mathbf{u}) - N(\mathbf{u}))^{2} : \mathbf{u} = \{ x_{j} \} \text{ for some } j
\text{ or } \mathbf{u} = \{ y_{i} \} \text{ for some } i \bigl\} \nonumber \\
& = & m_{1}^{2} + \cdots m_{s}^{2} + n_{1}^{2} + \cdots n_{t}^{2}.
\end{eqnarray}
\end{proof}

\begin{theorem}\label{thm:5}
The function $\gamma(p) = c_{1}\alpha_{1}^{p} + \cdots + c_{\ell}\alpha_{\ell}^{p}$ is constant or strictly increasing
on $(0, \infty)$.
\end{theorem}

\begin{proof} Suppose $0 < p_{1} < p_{2}$. For each $k \in [\ell]$ we set
$\delta_{k} = \alpha_{k}^{p_{2}} - \alpha_{k}^{p_{1}}$. Notice that $\delta_{1} = 0$ because $\alpha_{1} = 1$
but $\delta_{k} > 0$ otherwise. For each $k \geq 2$ the function $g_{k}(p) = \alpha_{k}^{p} - \alpha_{k-1}^{p}$
is strictly increasing on $(0, \infty)$ because $1 \leq \alpha_{k-1} < \alpha_{k}$. As a result,
$\delta_{k} - \delta_{k-1} = (\alpha_{k}^{p_{2}} - \alpha_{k-1}^{p_{2}}) - (\alpha_{k}^{p_{1}} - \alpha_{k-1}^{p_{1}}) > 0$
for each $k \geq 2$. Thus
  \begin{align}\label{three}
   \gamma(p_{2}) - \gamma(p_{1}) = \sum_{k=1}^{\ell} c_{k} \delta_{k}
&= (c_{1} + c_{2} + \dots + c_{\ell}) \delta_{1} \nonumber \\
    & \qquad + (c_{2} + c_{3} + \cdots +c_{\ell}) (\delta_{2} - \delta_{1}) \nonumber \\
    & \qquad\qquad \ddots \nonumber \\
   & \qquad\qquad\qquad +
         (c_{\ell-1} + c_{\ell}) (\delta_{\ell - 1} - \delta_{\ell - 2}) \nonumber \\
   & \qquad\qquad\qquad\qquad + c_{\ell} (\delta_{\ell} - \delta_{\ell - 1})
  \end{align}
is non-negative by Theorem \ref{thm:4}. Now suppose that $\gamma$ is non-constant on $(0, \infty)$.
Then we may consider the largest $r \geq 2$ such that $c_{r} \not= 0$. By Theorem \ref{thm:4},
$c_{r} = c_{r} + c_{r+1} + \cdots + c_{\ell} > 0$. Hence the term
$(c_{r} + c_{r+1} + \cdots + c_{\ell})(\delta_{r} - \delta_{r-1})$, which appears on the
right side of (\ref{three}), is positive. All other terms on the right side of (\ref{three})
are non-negative by Theorem \ref{thm:4}. Thus $\gamma(p_{2}) - \gamma(p_{1}) > 0$ in this instance.
\end{proof}

\begin{corollary}\label{l:bounds}
$\lim\limits_{p \rightarrow 0^{+}} \gamma_{p}(\boldsymbol{\omega})
\geq \vartheta (n)$. In particular, $\Gamma_{X}(p) \geq \vartheta (n)$ for all $p \geq 0$.
\end{corollary}

\begin{proof} Recall that $n = |X| > 1$ and $\vartheta(n) \defeq
\frac{1}{2} \left( \left\lfloor \frac{n}{2} \right\rfloor^{-1} + \left\lceil \frac{n}{2} \right\rceil^{-1} \right)$.

The right side of (\ref{two}) is subject to the constraint
$m_{1} + \cdots + m_{s} = 1 = n_{1} + \cdots + n_{t}$ and is easily seen to be at least $s^{-1} + t^{-1}$
by elementary calculus. Moreover, $s^{-1} + t^{-1} \geq 2 \vartheta(n)$ because $s + t \leq n$.
As a result, $$\lim\limits_{p \rightarrow 0^{+}} \gamma_{p}(\boldsymbol{\omega}) = c_{1} + \cdots + c_{\ell}
\geq \vartheta(n).$$ As this inequality holds for any $\boldsymbol{\omega} \in \mathcal{N}$ it
follows from Theorem \ref{thm:5} that $\Gamma_{X}(p) \geq \vartheta (n)$ for all $p \geq 0$.
\end{proof}

It is worth noting that the lower bound on $\Gamma_{X}(p)$ given in the statement of Corollary \ref{l:bounds}
is best possible. In the case of the discrete metric $d$ on $X$ one has $\Gamma_{X}(p) = \vartheta (n)$ for all
$p \geq 0$ by Weston \cite[Theorem 3.2]{We2}.

\section{Flat simplices and the computation of $\Gamma_{X}(\infty)$}\label{sec:5}

Let $(X,d) = (\{ z_{1}, \ldots, z_{n} \}, d)$, $\mathcal{N}$ and $\mathcal{D}$ be as in Section \ref{sec:4}.
We begin by showing that $\Gamma_{X}(\infty)$ exists in the real line.

\begin{theorem}\label{limit:2}
$\Gamma_{X}(p)$ is non-decreasing and bounded above by 1, and hence
$\Gamma_{X}(\infty) \defeq \lim\limits_{p \rightarrow \infty} \Gamma_{X}(p)$ exists.
\end{theorem}

\begin{proof} By Theorem \ref{gap:evl} there is simplex $\gamma_{p}(\boldsymbol{\omega}_{0})$
such that $\Gamma_{X}(p) = \gamma(p):=\gamma_{p}(\boldsymbol{\omega}_{0})$. So, by Theorem \ref{thm:5},
$\Gamma_{X}(p)$ is non-decreasing on $[0,\infty)$. To complete the proof we need only show that $\Gamma_{X}(p)$
is bounded above independently of $p$. This is easy. We may choose distinct $x, y \in X$ with $d(x, y) = 1$.
Then the simplex $\boldsymbol{\omega}_{1} = [x(1); y(1)]_{1,1}$ satisfies $\gamma_{p}(\boldsymbol{\omega}_{1}) = 1$
for all $p > 0$. Hence, by Theorem \ref{S5}, $\Gamma_{X}(p) \leq 1$ for all $p > 0$.
\end{proof}

\begin{remark}\label{rem:1}
Theorem \ref{limit:2} is specific to finite metric spaces that are ultrametric.
No such statement holds for finite metric spaces that are not ultrametric. Indeed, if $(Y, \rho)$ is a finite
metric space that is not ultrametric, then there is a real number $p_{0} > 0$
such that $\Gamma_{Y}(p_{0}) = 0$. This follows from Faver \textit{et al}.\ \cite[Corollary 5.3]{Fav}, and Li
and Weston \cite[Corollary 4.3]{Hli}. As $\Gamma_{Y}(0) > 0$ by Weston \cite[Theorem 3.2]{We2}, we see that
$\Gamma_{Y}(p)/ \alpha^{p}$, where $\alpha$ is the minimum non-zero distance in $(Y, \rho)$, cannot be
non-decreasing on $[0, p_{0}]$.
\end{remark}

We will see presently that the value of $\Gamma_{X}(\infty)$ is determined by the following class
of simplices.

\begin{definition}\label{flat:s}
A simplex $\boldsymbol{\omega}$ in $(X,d)$ is said to be \textit{flat} if
$\gamma(p) \defeq \gamma_{p}(\boldsymbol{\omega})$ is constant on $(0, \infty)$.
\end{definition}

We let $\mathcal{K} \defeq
\{ \boldsymbol{\omega} \in \mathcal{N} : \boldsymbol{\omega} \text{ is a flat simplex in } (X,d) \}$.
The set $\mathcal{K}$ plays a key role in computing $\Gamma_{X}(\infty)$.

\begin{lemma}
$\mathcal{K}$ is a closed set in $\mathbb{R}^{n}$.
\end{lemma}

\begin{proof}
By Theorem \ref{thm:3}, $\boldsymbol{\omega} \in \mathbb{R}^{n}$ belongs to $\mathcal{K}$ if and only if
$\sum \omega_{j} = 0$, $\sum |\omega_{j}| = 2$, $M(\mathbf{u}) = N(\mathbf{u})$ for each $\mathbf{u} \in \Pi_{1}$,
and all non-coterie weights are $0$. So $\boldsymbol{\omega}$ needs to satisfy a finite number of equations.
Thus $\mathcal{K}$ is closed in $\mathbb{R}^{n}$.
\end{proof}

\begin{theorem}\label{p=1}
$\Gamma_{X}(\infty) = \inf\limits_{\boldsymbol{\omega} \in \mathcal{K}} \gamma_{1}(\boldsymbol{\omega})$.
\end{theorem}

\begin{proof} First of all notice that
by Theorem \ref{S5} and Definition \ref{flat:s} we have
\begin{eqnarray}\label{up:bnd}
\Gamma_{X}(\infty)
&  =   & \lim\limits_{p \rightarrow \infty} \Gamma_{X}(p) \nonumber \\
&  =   & \lim\limits_{p \rightarrow \infty} \inf\limits_{\boldsymbol{\omega}
         \in \mathcal{N}} \gamma_{p}(\boldsymbol{\omega}) \nonumber \\
& \leq & \lim\limits_{p \rightarrow \infty} \inf\limits_{\boldsymbol{\omega}
         \in \mathcal{K}} \gamma_{p}(\boldsymbol{\omega}) \nonumber \\
&  =   & \inf\limits_{\boldsymbol{\omega} \in \mathcal{K}} \gamma_{1}(\boldsymbol{\omega}).
\end{eqnarray}
The switch to $p = 1$ in the last line of (\ref{up:bnd}) occurs because $\mathcal{K}$ is the collection of flat simplices.
We may choose an increasing sequence of real numbers $(p_{k})$ with $p_{1} = 1$ and $\lim p_{k} = \infty$.
For each $k \in \mathbb{N}$ we may, by Theorem \ref{gap:evl}, choose a simplex
$\boldsymbol{\omega}_{k} \in \mathcal{N}$ so that $\Gamma_{X}(p_{k}) = \gamma_{p_{k}}(\boldsymbol{\omega}_{k})$.
Now $(\boldsymbol{\omega}_{k})$ has a convergent subsequence because $\mathcal{N}$ is a compact metric space.
Without loss of generality we may assume that we have passed to such a subsequence and that the limit is
$\boldsymbol{\omega}_{\infty} \in \mathcal{N}$.

We claim that $\boldsymbol{\omega}_{\infty} \in \mathcal{K}$. If we assume, to the contrary, that $\boldsymbol{\omega}_{\infty}
\notin \mathcal{K}$, then $\gamma_{p}(\boldsymbol{\omega}_{\infty})$ is a strictly increasing function of $p$
(whose limit at $\infty$ is $\infty$) by Theorem \ref{thm:5}. Thus, by continuity, there exists a $q \in (0, \infty)$
such that $\gamma_{q}(\boldsymbol{\omega}_{\infty}) = 2$. Moreover, as $\gamma_{q}(\cdot)$ is continuous in
$\boldsymbol{\omega}$, there must exist some ball $B \subset \mathcal{N} \setminus \mathcal{K}$, centered at
$\boldsymbol{\omega}_{\infty}$ with positive radius, such that $\gamma_{q}(\boldsymbol{\omega}) > 1$
for all $\boldsymbol{\omega} \in B$. Then, by Theorem \ref{thm:5}, $\gamma_{p}(\boldsymbol{\omega}) > 1$
for all $p \geq q$ and all $\boldsymbol{\omega} \in B$.
As $(\boldsymbol{\omega}_{k})$ converges to $\boldsymbol{\omega}_{\infty}$, we may choose an integer
an $M > 0$ so that $(\boldsymbol{\omega}_{k})_{k \geq M} \subset B$. It follows that
$\gamma_{p_{k}}(\boldsymbol{\omega}_{k}) > 1$ for all sufficiently large $k$. However, we have chosen each
$\boldsymbol{\omega}_{k}$ so that $\gamma_{p_{k}}(\boldsymbol{\omega}_{k}) = \Gamma_{X}(p_{k})$.
This leads to a contradiction because each $\Gamma_{X}(p_{k}) \leq 1$ by Theorem \ref{limit:2}. We conclude
that $\boldsymbol{\omega}_{\infty} \in \mathcal{K}$.

By definition of $\mathcal{K}$, $\gamma_{1}(\boldsymbol{\omega}_{\infty}) = \gamma_{p}(\boldsymbol{\omega}_{\infty})$
for all $p > 0$. We claim that
\begin{eqnarray}\label{1:lim}
\lim\limits_{k \rightarrow \infty} \gamma_{p_{k}}(\boldsymbol{\omega}_{k}) & = & \gamma_{1}(\boldsymbol{\omega}_{\infty}).
\end{eqnarray}
Let $\epsilon > 0$ be given. The function $\gamma_{1}(\cdot)$ is continuous in $\boldsymbol{\omega}$. Hence there
exists an $r > 0$ such that if $\boldsymbol{\omega}$ is in the ball $B_{\boldsymbol{\omega}_{\infty}}(r)$, then
$| \gamma_{1}(\boldsymbol{\omega}_{\infty}) - \gamma_{1}(\boldsymbol{\omega}) | < \epsilon$ and so, in particular,
$\gamma_{1}(\boldsymbol{\omega}_{\infty}) - \gamma_{1}(\boldsymbol{\omega}) < \epsilon$. For each fixed
$\boldsymbol{\omega} \in \mathcal{N}$, $\gamma_{p}(\boldsymbol{\omega})$ is a non-decreasing function of $p$ by
Theorem \ref{thm:5}. Thus, for all $p \geq 1$ and all $\boldsymbol{\omega} \in B_{\boldsymbol{\omega}_{\infty}}(r)$,
we see that $\gamma_{1}(\boldsymbol{\omega}_{\infty}) - \gamma_{p}(\boldsymbol{\omega}) \leq
\gamma_{1}(\boldsymbol{\omega}_{\infty}) - \gamma_{1}(\boldsymbol{\omega}) < \epsilon$. As
$(\boldsymbol{\omega}_{k})$ converges to $\boldsymbol{\omega}_{\infty}$, we may choose an integer $N > 0$ so
that $(\boldsymbol{\omega}_{k})_{k \geq N} \subset B_{\boldsymbol{\omega}_{\infty}}(r)$. Hence, for $k \geq N$, we have
$\gamma_{1}(\boldsymbol{\omega}_{\infty}) - \gamma_{p_{k}}(\boldsymbol{\omega}_{k}) \leq
\gamma_{1}(\boldsymbol{\omega}_{\infty}) - \gamma_{1}(\boldsymbol{\omega}_{k}) < \epsilon$. On the other hand,
each $\boldsymbol{\omega}_{k}$ is chosen so that
\begin{eqnarray*}
\gamma_{p_{k}}(\boldsymbol{\omega}_{k}) = \Gamma_{X}(p_{k}) = \inf\limits_{\boldsymbol{\omega} \in \mathcal{N}}
\gamma_{p_{k}}(\boldsymbol{\omega}).
\end{eqnarray*}
Hence $\gamma_{1}(\boldsymbol{\omega}_{\infty}) - \gamma_{p_{k}}(\boldsymbol{\omega}_{k}) =
\gamma_{p_{k}}(\boldsymbol{\omega}_{\infty}) - \gamma_{p_{k}}(\boldsymbol{\omega}_{k}) \geq 0$. So we have that
$0 \leq \gamma_{1}(\boldsymbol{\omega}_{\infty}) - \gamma_{p_{k}}(\boldsymbol{\omega}_{k}) < \epsilon$ for all
$k \geq N$. As $\epsilon > 0$ was arbitrary, this establishes (\ref{1:lim}). Thus,
\begin{eqnarray}\label{low:bnd}
\Gamma_{X}(\infty)
&  =   & \lim\limits_{p \rightarrow \infty} \Gamma_{X}(p) \nonumber \\
&  =   & \lim\limits_{k \rightarrow \infty} \Gamma_{X}(p_{k}) \nonumber \\
&  =   & \lim\limits_{k \rightarrow \infty} \gamma_{p_{k}}(\boldsymbol{\omega}_{k}) \nonumber \\
&  =   & \gamma_{1}(\boldsymbol{\omega}_{\infty}) \nonumber \\
& \geq & \inf\limits_{\boldsymbol{\omega} \in \mathcal{K}} \gamma_{1}(\boldsymbol{\omega}).
\end{eqnarray}
The theorem now follows from (\ref{up:bnd}) and (\ref{low:bnd}).
\end{proof}

\begin{theorem}\label{X}
Let $B_{1}, B_{2}, \ldots, B_{l}$ denote the distinct coteries of $(X,d)$. Then,
\begin{eqnarray*}
\Gamma_{X}(\infty) & = &
\left\{ \vartheta (|B_{1}|)^{-1} + \cdots + \vartheta (|B_{l}|)^{-1} \right\}^{-1}.
\end{eqnarray*}
\end{theorem}

\begin{proof}
For each coterie $B_{k}$ there is a unique node $\mathbf{v}_{k} \in \Pi_{1}$ such that $B_{k} = \Adj (\mathbf{v}_{k})$,
and conversely. Moreover, the inherited metric on each coterie $B_{k}$ is simply the discrete metric on $B_{k}$.
This is because we are assuming that the minimum non-zero distance in our finite ultrametric space $(X,d)$ is
$\alpha_{1} = 1$. As a result, $\Gamma_{B_{k}}(1) = \Gamma_{B_{k}}(0) = \vartheta(|B_{k}|)$ for each $k \in [l]$ by
Weston \cite[Theorem 3.2]{We2}. This fact will be used below.

By Theorem \ref{thm:3}, $\boldsymbol{\omega} = [x_{j}(m_{j}); y_{i}(n_{i})]_{s,t} \in \mathcal{K}$
if and only if $M(\mathbf{v}_{k}) = N(\mathbf{v}_{k})$
for each $k \in [l]$ and no $x_{j}$ or $y_{i}$ belongs to $X \setminus (B_{1} \cup \cdots \cup B_{l})$.
Now consider a fixed $\boldsymbol{\omega} = (\omega_{r}) = [x_{j}(m_{j}); y_{i}(n_{i})]_{s,t} \in \mathcal{K}$.
For each $k \in [l]$ we form a new vector $\boldsymbol{\omega}_{k}$ from $\boldsymbol{\omega}$ in the following way:
if $z_{r} \notin B_{k}$ redefine $\omega_{r}$ to be $0$. Otherwise, make no change to $\omega_{r}$.
While it may not be the case that $\boldsymbol{\omega}_{k} \in \mathcal{N}$ we can still define
$\gamma_{1}(\boldsymbol{\omega}_{k})$ according to the formula given in Definition \ref{S4}.
Now $2 \cdot \gamma_{1}(\boldsymbol{\omega}) = m_{1}^{2} + \cdots + m_{s}^{2} + n_{1}^{2} + \cdots + n_{t}^{2}$
by the second part of Theorem \ref{thm:4}. Moreover, for each $k \in [l]$, the argument used to prove Lemma \ref{lem:1} may be
easily modified to show that
\begin{eqnarray*}
2 \cdot \gamma_{1}(\boldsymbol{\omega}_{k}) & = & \sum\limits_{j: x_{j} \in B_{k}} m_{j}^{2}
+ \sum\limits_{i: y_{i} \in B_{k}} n_{i}^{2}.
\end{eqnarray*}
Thus $\gamma_{1}(\boldsymbol{\omega}) = \gamma_{1}(\boldsymbol{\omega}_{1}) + \cdots + \gamma_{1}(\boldsymbol{\omega}_{l})$.
Provided $w_{k} \defeq M(\mathbf{v}_{k}) = N(\mathbf{v}_{k}) > 0$,
we may further set $\boldsymbol{\upsilon}_{k} \defeq \boldsymbol{\omega}_{k} / w_{k}$.
(In the event that $w_{k} = 0$ it suffices to let $\boldsymbol{\upsilon}_{k}$ be any simplex in $B_{k}$.)
Then $\boldsymbol{\upsilon}_{k}$ is a simplex in $B_{k}$, $\boldsymbol{\upsilon}_{k} \in \mathcal{K}$ and
$\gamma_{1}(\boldsymbol{\omega}_{k}) = w_{k}^{2} \gamma_{1}(\boldsymbol{\upsilon}_{k})$ by construction.
Thus, $\gamma_{1}(\boldsymbol{\omega}) = w_{1}^{2} \gamma_{1}(\boldsymbol{\upsilon}_{1}) + \cdots +
w_{l}^{2} \gamma_{1}(\boldsymbol{\upsilon}_{l})$.

By Theorem \ref{p=1},
\begin{eqnarray*}
\Gamma_{X}(\infty)
& = & \inf\limits_{\boldsymbol{\omega} \in \mathcal{K}} \gamma_{1}(\boldsymbol{\omega}) \\
& = & \inf\limits_{\boldsymbol{\omega} \in \mathcal{K}}
\gamma_{1}(\boldsymbol{\omega}_{1}) + \cdots + \gamma_{1}(\boldsymbol{\omega}_{l}) \\
& = & \inf\limits_{\boldsymbol{\omega} \in \mathcal{K}} w_{1}^{2} \gamma_{1}(\boldsymbol{\upsilon}_{1}) + \cdots +
w_{l}^{2} \gamma_{1}(\boldsymbol{\upsilon}_{l}) \\
& = & \inf \left\{ w_{1}^{2} \gamma_{1}(\boldsymbol{\varsigma}_{1}) + \cdots + w_{l}^{2} \gamma_{1}(\boldsymbol{\varsigma}_{l}) :
\boldsymbol{\varsigma}_{k} \text{ is a simplex in } B_{k} \text{ and } w_{1} + \cdots + w_{l} = 1 \right\} \\
& = & \inf \left\{ w_{1}^{2} \Gamma_{B_{1}}(1) + \cdots + w_{l}^{2} \Gamma_{B_{l}}(1) : w_{1} + \cdots + w_{l} = 1 \right\} \\
& = & \inf \left\{ w_{1}^{2} \vartheta(|B_{1}|) + \cdots + w_{l}^{2} \vartheta(|B_{l}|) : w_{1} + \cdots + w_{l} = 1 \right\} \\
& = & \left\{ \vartheta (|B_{1}|)^{-1} + \cdots + \vartheta (|B_{l}|)^{-1} \right\}^{-1}.
\end{eqnarray*}
The last equality follows from an application of Lagrange's multiplier theorem.
\end{proof}

Theorem \ref{main} and Corollary \ref{main:cor} now follow directly from (\ref{Y}) and Theorem \ref{X}.

\begin{remark}\label{ian:ex}
We conclude this section with a brief discussion of a specific example.
Even in the case of a finite ultrametric space $(X,d)$ with relatively few points and distinct non-zero
distances, determining an explicit expression for $\Gamma_{X}(p)$ can be a surprisingly difficult task.
However, in small spaces with just two non-zero distances one can often explicitly calculate $\Gamma_{X}(p)$.

Consider, for example, the ultrametric $d$ induced on the set $X = \{ z_{1}, z_{2}, z_{3}, z_{4}, z_{5}, z_{6} \}$
by the dendrogram with proximity part $\{ 0 = \alpha_{0}, 1 = \alpha_{1}, \alpha_{2} \}$
and partition function $\pi (1) = \{ \{ z_{1} \}, \{ z_{2}, z_{3} \}, \{z_{4}, z_{5}, z_{6} \} \}$.
In this case one can calculate directly (by using Lagrange's multiplier theorem or formulas of Wolf \cite{Wol})
that for any $p \geq 0$, we have
\begin{eqnarray*}
\Gamma_{X}(p) & = & \frac{9 \alpha_{2}^{2p} - 7 \alpha_{2}^{p} + 1}{21 \alpha_{2}^{2p} - 12 \alpha_{2}^{p}}.
\end{eqnarray*}
Taking the limit as $p \rightarrow \infty$, we see that
\[
\Gamma_{X}(\infty)  =  \frac{3}{7} = (\vartheta(2)^{-1} + \vartheta(3)^{-1})^{-1},
\]
in agreement with Theorem \ref{X}.
Notice in this instance that we have $\Gamma_{X}(\infty) - \Gamma_{X}(p) = \mathcal{O}(\alpha_{2}^{-p})$.
In general, we do not know the rate of convergence of the limit in Theorem \ref{X}.
\end{remark}

\section{Determining when $\Gamma_{X}(p)$ is constant on $[0, \infty)$}\label{sec:6}

Let $(X,d)$ be a finite ultrametric space with $|X| > 1$ and minimum non-zero distance $\alpha_{1} = 1$.
In this section we determine when the $p$-negative type gap $\Gamma_{X}(p)$ is constant on
$[0, \infty)$. It is helpful to note some basic properties of the function $\vartheta (n)$ (see (\ref{theta})).
If $n > 1$ is even, then $\vartheta (n) = 2/n$, and so $\vartheta (n)^{-1} - n/2 = 0$.
On the other hand, if $n > 1$ is odd, then $\vartheta (n) = 2n/(n^{2} - 1)$, and so $\vartheta (n)^{-1} - n/2 = -(1/2n)$.

\begin{corollary}\label{constant}
Let $l \geq 1$ be an integer. If $n \defeq n_{1} + \cdots + n_{l}$ where $n_{1}, \ldots, n_{l} > 1$ are integers, then
\begin{eqnarray}\label{ineq:one}
\vartheta (n) & \leq & \left\{ \vartheta (n_{1})^{-1} + \cdots + \vartheta (n_{l})^{-1} \right\}^{-1}.
\end{eqnarray}
Moreover, we have equality if and only if $l = 1$ or $l > 1$ and all of the integers $n_{1}, \ldots, n_{l}$ are even.
\end{corollary}

\begin{proof}
Inequality (\ref{ineq:one}) follows from the calculations in this paper. Indeed, given integers
$n_{1}, \ldots, n_{l} > 1$, consider any finite ultrametric space $(X,d)$
with minimum non-zero distance $\alpha_{1} = 1$ and distinct coteries $B_{1}, \ldots, B_{l}$ such that
$n = |X| = |B_{1}| + \cdots + |B_{l}|$ and $n_{j} = |B_{j}|$ for each $j \in [l]$. Then,
by Corollary \ref{l:bounds} and Theorem \ref{X}, we see that $\vartheta (n) = \Gamma_{X}(0) \leq \Gamma_{X}(\infty) =
\left\{ \vartheta (n_{1})^{-1} + \cdots + \vartheta (n_{l})^{-1} \right\}^{-1}.$

If $l = 1$ the inequality (\ref{ineq:one}) is, trivially, an equality. Suppose that $l > 1$.
Then we have $1 < n_{j} < n$ for each $j \in [l]$. Notice that the inequality
\begin{eqnarray}\label{ineq:two}
\vartheta (n)^{-1} - \frac{n}{2}
& \geq & \sum\limits_{j=1}^{l} \left( \vartheta (n_{j})^{-1} - \frac{n_{j}}{2} \right)
\end{eqnarray}
is an equivalent formulation of (\ref{ineq:one}) because $n = n_{1} + \cdots + n_{l}$. The parenthetical
terms on the right side of (\ref{ineq:two}) are $0$ (if $n_{j}$ is even) or negative (if $n_{j}$ is odd).
The right side of (\ref{ineq:two}) is therefore $0$ or negative. We now distinguish between two cases.

If $n$ is even, $\vartheta (n)^{-1} - n/2 = 0$. On the other hand, right side of (\ref{ineq:two})
is negative if and only if at least one $n_{j}$ is odd. So, in this case, we have equality in the
inequality (\ref{ineq:one}) if and only if each $n_{j}$ is even.

If $n$ is odd, $\vartheta (n)^{-1} - n/2 = -(1/2n)$, and at least one $n_{j}$ is odd. So, in this
case, we have
\[
\sum\limits_{j=1}^{l} \left( \vartheta (n_{j})^{-1} - \frac{n_{j}}{2} \right) =
\sum\limits_{j : n_{j} \text{ odd}} \frac{-1}{2n_{j}} < \frac{-1}{2n},
\]
because $n_{j} < n$ for each $j \in [l]$. This means that equality is not possible
in inequality (\ref{ineq:one}) if $n$ is odd.

The statement concerning the case of equality in inequality (\ref{ineq:one}) is now clear.
\end{proof}

It is important to note that if $d$ is the discrete metric on $X$,
then we have $\Gamma_{X}(p) = \Gamma_{X}(0) = \vartheta (|X|)$ for all $p \geq 0$. This follows from
Weston \cite[Theorem 3.2]{We2}. If we assume that $d$ is not the discrete metric on $X$ (so that
$\alpha_{1} = 1$ is not the only non-zero distance in $(X,d)$),
then $(X,d)$ will have at least two coteries or we will have $|X| < |B_{1}| + \cdots |B_{l}|$ where
$B_{1}, B_{2}, \ldots, B_{l}$ are the distinct coteries of $(X,d)$. However, if it is the case that
$|X| < |B_{1}| + \cdots |B_{l}|$, then $\Gamma_{X}(0) = \vartheta (|X|) < \Gamma_{X}(\infty)$.
This is because $\vartheta (n)$ decreases strictly as $n$ increases and
because $\Gamma_{X}(\infty) \geq \vartheta (|B_{1}| + \cdots |B_{l}|)$ by Corollary \ref{constant}.
These comments support the proof of the main result of this section.

\begin{theorem}\label{con:thm}
If $B_{1}, B_{2}, \ldots, B_{l}$ are the distinct coteries of $(X,d)$ and $d$ is not
the discrete metric on $X$, then the following conditions are equivalent:
\begin{enumerate}
\item $\Gamma_{X}(0) = \Gamma_{X}(\infty)$

\item $\Gamma_{X}(p)$ is constant on $[0, \infty)$

\item $|X| = |B_{1}| + \cdots |B_{l}|$ and all of the integers $|B_{1}|, \ldots, |B_{l}|$ are even.
\end{enumerate}
\end{theorem}

\begin{proof}
The equivalence of (1) and (2) is plain because $\Gamma_{X}(p)$ is non-decreasing on $[0, \infty)$.
The equivalence of (1) and (3) follows from Corollary \ref{constant}.
\end{proof}

Theorem \ref{main2} now follows directly from (\ref{Y}) and Theorem \ref{con:thm}.

\section*{Acknowledgements}
We would like to thank the CCRDS research group at the University of South Africa for their
kind support and insightful comments during the preparation of this paper. The last named author
was partially supported by the \textit{Visiting Researcher Support Programme} at the University
of South Africa and was provided additional support by the University of New South Wales.

\bibliographystyle{amsalpha}

\end{document}